\documentclass[10.5pt,a4paper]{article} 

\usepackage{graphicx,latexsym,euscript,makeidx,color}
\usepackage{amsmath,amsfonts,amssymb,amsthm,mathrsfs}
\usepackage{enumerate}

\usepackage[colorlinks,linkcolor=blue,anchorcolor=green,citecolor=red]{hyperref}%when print, use the next package
\usepackage{geometry}
\def\5n{\negthinspace \negthinspace \negthinspace \negthinspace \negthinspace }
\def\4n{\negthinspace \negthinspace \negthinspace \negthinspace }
\def\3n{\negthinspace \negthinspace \negthinspace }
\def\2n{\negthinspace \negthinspace }
\def\1n{\negthinspace }

   \def\cF{{\cal F}}  
     
\def\dbH{\mathbb{H}}     
   \def\cI{{\cal I}}

 \def\sL{\mathscr{L}}

\def\dbR{\mathbb{R}}     
\def\dbS{\mathbb{S}}     
     
   \def\cU{{\cal U}}  
\def\dbV{\mathbb{V}}     
     
\def\dbX{\mathbb{X}}

\def\ss{\smallskip}                \def\lt{\left}
\def\ms{\medskip}                \def\rt{\right}
               \def\lan{\langle}
\def\ds{\displaystyle}           \def\ran{\rangle}
\def\no{\noindent}        \def\q{\quad}                      
\def\ns{\noalign{\ss}}    \def\qq{\qquad}                    
    \def\hb{\hbox}                     \def\blan{\big\langle}
                   \def\bran{\big\rangle}
         \def\rf{\eqref}                    \def\Blan{\Big\langle}
  \def\deq{\triangleq}               \def\Bran{\Big\rangle}
 \def\ae{\hbox{\rm a.e.}}           \def\({\Big (}
\def\les{\leqslant}                  \def\){\Big )}
\def\ges{\geqslant}          \def\[{\Big[}
           \def\]{\Big]}
          \def\tr{\hbox{\rm tr$\,$}}         \def\cd{\cdot}
              \def\cds{\cdots}

\def\a{\alpha}        \def\G{\Gamma}         
            \def\d{\delta}   \def\F{\Phi}     
             \def\Si{\Sigma}  \def\si{\sigma}
     \def\l{\lambda}        
    \def\t{\tau}     \def\f{\varphi}  \def\i{\infty}   

\def\ba{\begin{array}}                \def\ea{\end{array}}
\def\bel{\begin{equation}\label}      \def\ee{\end{equation}}
\def\ben{\begin{enumerate}}           \def\een{\end{enumerate}}

\newtheorem{theorem}{\indent Theorem}[section]
\newtheorem{definition}[theorem]{\indent Definition}
\newtheorem{proposition}[theorem]{\indent Proposition}

\newtheorem{lemma}[theorem]{\indent Lemma}
\newtheorem{remark}[theorem]{\indent Remark}
\newtheorem{example}[theorem]{\indent Example}

\makeatletter
   
   \@addtoreset{equation}{section}
\makeatother

\allowdisplaybreaks[4]
\begin{document}

\title{\bf Linear Quadratic Optimal Control Problems \\ with Fixed Terminal States and Integral
Quadratic Constraints}
\author{Jingrui Sun\thanks{Department of Mathematics, National University of Singapore, 119076,
Singapore (matsunj@nus.edu.sg).}}
%\date{}
\maketitle

%--------------------------------------------------------------------------------------------------------

\no\bf Abstract. \rm
This paper is concerned with a linear quadratic (LQ, for short) optimal control problem
with fixed terminal states and integral quadratic constraints.
A Riccati equation with infinite terminal value is introduced, which is uniquely solvable
and whose solution can be approximated by the solution for a suitable unconstrained LQ
problem with penalized terminal state.
Using results from duality theory, the optimal control is explicitly derived by solving
the Riccati equation together with an optimal parameter selection problem.
It turns out that the optimal control is not only a feedback of the current state,
but also a feedback of the target (terminal state).
Some examples are presented to illustrate the theory developed.

\ms

\no\bf Key words. \rm
linear quadratic optimal control, constraint, complete controllability, Riccati equation,
feedback

\ms

\no\bf AMS subject classifications. \rm 49J15, 49N10, 49N35, 93B05

\section{Introduction}

Linear quadratic (LQ, for short) problems constitute an extremely important class of optimal control problems.
They are widely encountered in many fields, such as engineering, economy, and biology, and also play an essential
role in the study of general optimal control problems.
The LQ problems have been extensively investigated since the earliest work of Bellman, Glicksberg, and Gross \cite{Bellman-Glicksberg-Gross 1958}, Kalman \cite{Kalman 1960}, and Letov \cite{Letov 1961}, however, very few
studies actually involve constraints on both the state and control variables.
There is no doubt that it is a much more challenging and interesting task to solve an LQ problem with constraints
than one without, and that developing a deeper understanding of constrained LQ problems, as well as efficient
algorithms for solving them, will have a big impact in a number of applications.

\ms

The aim of this paper is to study a class of constrained LQ optimal control problems whose main features are
that the state end-points are fixed and that there are integral quadratic constraints.
To be precise, consider the controlled linear system on a finite horizon $[t,T]$:
\bel{state}\lt\{\2n\ba{ll}
\ds\dot X(s)=A(s)X(s)+B(s)u(s),\qq s\in[t,T],\\
\ns\ds X(t)=x. \ea\rt.\ee
A control $u(\cd)$ is called {\it admissible} if $u(\cd)\in L^2(t,T;\dbR^m)\equiv\cU[t,T]$,
the space of all $\dbR^m$-valued functions that are square-integrable on $[t,T]$.
Assuming the system \rf{state} is {\it completely controllable} on $[t,T]$, we know that
for each initial state $x$ and each target $y$, there exist admissible controls $u(\cd)$ giving $X(T)=y$.
For $(t,x,y)\in[0,T)\times\dbR^n\times\dbR^n$, we denote the corresponding solution of \rf{state}
by $X(\cd\,;t,x,u(\cd))$ and define
$$\cU(t,x,y)=\big\{u:[t,T]\to\dbR^m~|~u(\cd)\in\cU[t,T]\hb{ and }
X(T;t,x,u(\cd))=y\big\}.$$
For any $(t,x,y)\in[0,T)\times\dbR^n\times\dbR^n$ and any $u(\cd)\in\cU(t,x,y)$,
the associated cost ($i=0$) and constraint functionals ($i=1,\ldots,k$) are given by
\bel{J_i}J_i(t,x,y;u(\cd))=\int_t^T\[\lan Q_i(s)X(s),X(s)\ran+\lan R_i(s)u(s),u(s)\ran\]ds,\ee
where $Q_i(\cd)$, $R_i(\cd)$, $i=0,1,\ldots,k$ are symmetric positive semi-definite matrices of
proper dimensions. Now given constants $c_1,\ldots,c_k>0$, the constrained LQ optimal control
problem considered in this paper can be stated as follows:

\ms

\bf Problem (CLQ). \rm
For any given initial pair $(t,x)\in[0,T)\times\dbR^n$ and any given target $y\in\dbR^n$,
find an admissible control $u^*(\cd)$ such that the cost functional $J_0(t,x,y;u(\cd))$
is minimized over $\cU[t,T]$, subject to the terminal state and functional constraints
\bel{constraint}X(T;t,x,u(\cd))=y, \q J_i(t,x,y;u(\cd))\les c_i;\qq i=1,\ldots,k.\ee

\ms

Any admissible control $u(\cd)$ satisfying the constraints \rf{constraint} is called a
{\it feasible} control (w.r.t. $(t,x,y)$), and it is called {\it strictly feasible}
(w.r.t. $(t,x,y)$) if the inequalities in \rf{constraint} are strict. A feasible control
is called {\it optimal} (w.r.t. $(t,x,y)$) if it solves Problem (CLQ) for the initial pair
$(t,x)$ and the target $y$. The infimum
$$V(t,x,y)\deq \inf\{J_0(t,x,y;u(\cd)): u(\cd)\hb{~is feasible w.r.t.~}(t,x,y)\}$$
is called the {\it value function} of Problem (CLQ).

\ms

The study of LQ optimal control problems has a long history that can be traced back to the work of Bellman,
Glicksberg, and Gross \cite{Bellman-Glicksberg-Gross 1958} in 1958, Kalman \cite{Kalman 1960} in 1960,
and Letov \cite{Letov 1961} in 1961. Since then, many researchers have made contributions to such kind of
problems and applications; see, for example, Geerts and Hautus \cite{Geerts-Hautus 1990},
Jurdgevic \cite{Jurdgevic 1990}, Jurdgevic and Kogan \cite{Jurdgevic-Kogan 1989}, Willems, Kitap\c{c}i,
and Silverman \cite{Willems-Kitapci-Silverman 1986}, and Yakubovich \cite{Yakubovich 1984}. For a thorough
study of unconstrained LQ problems, we further refer the reader to the classical books of Anderson
and Moore \cite{Anderson-Moore 1971,Anderson-Moore 1989}, Lee and Markus \cite{Lee-Markus 1967},
Wonham \cite{Wonham 1985}, Yong and Zhou \cite{Yong-Zhou 1999}, and the survey paper of Willems \cite{Willems 1971}.

\ms

One of the elegant features of the LQ theory is that the optimal control can be explicitly
represented in a state feedback form, through the solution to the celebrated Riccati equation.
Hence, the LQ problem can be reduced to that of solving the Riccati equation.
Generally, there are three approaches for deriving the Riccati equation, namely the maximum principle,
the dynamic programming, and the completion of squares technique.
What essentially makes these approaches successful, besides the special LQ structure,
is that the problem is not constrained.
If there are state and control constraints, the whole LQ approach may collapse.

\ms

However, many applications of optimal control theory are constrained problems.
A typical example is flight planing in which the terminal state (destination) is fixed.
Flight planners normally wish to minimize flight cost through the appropriate choice of route,
height, and speed, and by loading the minimum necessary fuel on board.
To ensure that the aircraft can safely reach the destination limits in a given time,
strict performance specifications must be adhered to in all flying conditions,
which can be expressed in the form of integral quadratic constraints.
Other applications can be found in the problem of controlling certain space structures
\cite{Toivonen-Makila 1989} and portfolio selection \cite{Hu-Zhou 2005}.
There were some attempts in attacking the constrained LQ control problems;
see for example \cite{Friedland 1967,Curtain 1984,Emirsajlow 1987,Emirsajlow 1989,
Lim-Liu-Teo-Moore 1999,Lim-Zhou 1999}. However, none of these works and their
associated analyses actually involve constraints on both the state and control variables.
Therefore there is need for the development and analysis of efficient solution techniques
for constrained LQ control problems.

\ms

The main purpose of this paper is to give a complete solution to the LQ problem with
fixed terminal states and integral quadratic constraints.
The principal method for solving the problem is combination of duality theory
and approximation techniques.
We first approach the constrained LQ problem as a convex optimization problem.
By the Lagrangian duality, it turns out that the optimal control can be derived by
solving an LQ control problem with only a terminal state constraint together with an
optimal parameter selection problem.
We then approximate the reduced LQ problem, whose terminal state is fixed, by a sequence
of standard LQ problems with penalized terminal states.
This leads to the existence and uniqueness of a solution to the Riccati equation with
infinite terminal value.
With the solutions of the Riccati equations, we are able to calculate the gradient for
the cost functional of the optimal parameter selection problem, and therefore the optimal
control is obtained, which is a feedback of both the current state and the target.

\ms

The rest of the paper is organized as follows.
Section 2 collects some preliminaries. Among other things, we establish the unique
solvability of Problem (CLQ).
In Section 3, we present the main results of the paper (with their proofs deferred
to Section 5 and 6).
In Section 4, using duality theory, we reduce Problem (CLQ) to a parameterized LQ
problem with only one constraint on the terminal state, then approximate it by a
sequence of unconstrained LQ problems with penalized terminal states.
The existence and uniqueness theorem for the Riccati equation with infinite terminal
value is proved in Section 5.
Section 6 is devoted to the proof of the main result Theorem \ref{thm:main-2}.
Some examples are presented in section 7 to illustrate the results obtained.

\section{Preliminaries}

Throughout this paper, we will denote by $M^\top$ the transpose of a matrix $M$ and by $\tr(M)$ the
trace of $M$. Let $\dbR^{n\times m}$ be the Euclidean space consisting of $(n\times m)$ real matrices
and let $\dbR^n=\dbR^{n\times1}$. The inner product in $\dbR^{n\times m}$ is denoted by $\lan M,N\ran$,
where $M,N\in\dbR^{n\times m}$, so that $\lan M,N\ran=\tr(M^\top N)$. This induces the Frobenius norm
$|M|=\sqrt{\tr(M^\top M)}$. Denote by $\dbS^n$ the space of all symmetric $(n\times n)$ real matrices,
and by $\dbS^n_+$ the space of all symmetric positive definite $(n\times n)$ real matrices.
For $\dbS^n$-valued functions $M$ and $N$, if $M-N$ is positive (respectively, semi-) definite a.e.,
we write $M>N$ (respectively, $M\ges N$), and if there exists a $\d>0$ such that $M-N\ges\d I$ a.e,
we write $M\gg N$. Let $\cI$ be an interval and $\dbH$ a Euclidean space. We shall denote by $C(\cI;\dbH)$
the space of all $\dbH$-valued continuous functions on $\cI$, and by $L^p(\cI;\dbH)$ $(1\les p\les\i)$
the space of all $\dbH$-valued functions that are $p$th power Lebesgue integrable on $\cI$.

\ms

Throughout this paper, we impose the following assumption:

\ms

{\bf(H1)} The matrices appearing in \rf{state} and \rf{J_i} satisfy
$$\left\{\2n\ba{ll}
\ds    A(\cd)\in L^1(0,T;\dbR^{n\times n}), &~ B(\cd)\in L^2(0,T;\dbR^{n\times m}),\\
\ns\ds Q_i(\cd)\in L^1(0,T;\dbS^n),         &~ Q_i(\cd)\ges0,\\
\ns\ds R_i(\cd)\in L^\i(0,T;\dbS^m),        &~ R_i(\cd)\ges0, \q R_0(\cd)\gg0.
\ea\right.$$

\ss

Consider the controlled ordinary differential system
\bel{AB}\dot X(s)=A(s)X(s)+B(s)u(s),\ee
which we briefly denote by $[A,B]$.
For $0\les t_0<t_1\les T$, we denote $\cU[t_0,t_1]\equiv L^2(t_0,t_1;\dbR^m)$.
Clearly, under (H1), for any initial pair $(t_0,x)$ and any $u(\cd)\in\cU[t_0,t_1]$,
equation \rf{AB} admits a unique solution $X(\cd)\equiv X(\cd\,;t_0,x,u(\cd))\in C([t_0,t_1];\dbR^n)$.
We now introduce the following definition.

\begin{definition}\rm
System $[A,B]$ is called {\it completely controllable on $[t_0,t_1]$},
if for any $x,y\in\dbR^n$ there exists a $u(\cd)\in\cU[t_0,t_1]$ such that
$$X(t_1;t_0,x,u(\cd))=y.$$
System $[A,B]$ is just called {\it completely controllable} if it is completely controllable on
any subinterval $[t_0,t_1]$ of $[0,T]$.
\end{definition}

It is well known that system $[A,B]$ is completely controllable on $[t_0,t_1]$ if and only if
$$\int_{t_0}^{t_1}\F_A(s)^{-1}B(s)\big[\F_A(s)^{-1}B(s)\big]^\top ds>0,$$
where $\F_A(\cd)$ is the solution to the $\dbR^{n\times n}$-valued ordinary differential equation
(ODE, for short)
\bel{F_A}\lt\{\2n\ba{ll}
\ds \dot\F_A(s)=A(s)\F_A(s), \qq s\in[0,T],\\
\ns\ds \F_A(0)=I. \ea\rt.\ee
The latter in turn is equivalent to the following {\it regular} condition:
\bel{regular}\eta^\top\F_A(s)^{-1}B(s)=0\q\ae~s\in[t_0,t_1] \q\Longrightarrow\q \eta=0.\ee
In particular, when the matrices $A(\cd)$ and $B(\cd)$ are constant-valued (time-invariant),
the complete controllability of system $[A,B]$ can be verified by checking the {\it Kalman rank condition}
$$\hb{rank}\,(B,AB,\cds,A^{n-1}B)=n.$$

In the rest of the paper, we will assume the following so that
every target $y$ can be reached from an arbitrary initial pair $(t,x)$:

\ms

{\bf(H2)} System $[A,B]$ is completely controllable.

\ms

Now returning to Problem (CLQ), we have the following basic result
which is concerned with the existence of an optimal control.

\begin{theorem}\label{thm:uni+exi}\sl
Let {\rm (H1)--(H2)} hold, and let $(t,x,y)\in [0,T)\times\dbR^n\times\dbR^n$ be given.
Suppose the set of feasible controls w.r.t. $(t,x,y)$ is nonempty. Then Problem {\rm(CLQ)}
admits a unique solution.
\end{theorem}

\begin{proof}
Let $\cF(t,x,y)$ denote the set of feasible controls w.r.t. $(t,x,y)$, that is,
$$\cF(t,x,y)=\big\{u(\cd)\in L^2(t,T;\dbR^m):X(T;t,x,u(\cd))=y,~J_i(t,x,y;u(\cd))\les c_i;~i=1,\ldots,k\big\}.$$
Observing that the mappings
$$u(\cd)\mapsto X(T;t,x,u(\cd)),\q u(\cd)\mapsto J_i(t,x,y;u(\cd));\qq i=1,\ldots,k$$
are convex and continuous, one can easily verify that $\cF(t,x,y)$ is a convex closed subset of $L^2(t,T;\dbR^m)$.
Becaus $Q_0(\cd)\ges 0$ and $R_0(\cd)\ges\d I$ for some $\d>0$, the cost functional $J_0(t,x,y;\,\cd\,)$ defined
on $\cF(t,x,y)$ is strictly convex and continuous, and hence sequentially weakly lower semicontinuous
(see \cite[Theorem 7.2.6]{Kurdila-Zabarankin 2005}).
Let $\{u_k(\cd)\}_{k=1}^\i\subseteq\cF(t,x,y)$ be a minimizing sequence for $J_0(t,x,y;\,\cd\,)$.
Since $\cF(t,x,y)$ is nonempty, we have
$$\d\int_t^T|u_k(s)|^2ds\les J_0(t,x,y;u_k(\cd))\to V(t,x,y)<\i.$$
This implies that $\{u_k(\cd)\}_{k=1}^\i$  is bounded in the Hilbert space $L^2(t,T;\dbR^m)$.
Consequently, there exists a subsequence $\{u_{k_j}(\cd)\}_{j=1}^\i$ converging weakly to some
$u^*(\cd)\in L^2(t,T;\dbR^m)$. Since $\cF(t,x,y)$ is a convex and closed, it follows form
Mazur's lemma that $u^*(\cd)\in\cF(t,x,y)$. Thus, by the sequential weak lower semicontinuity of the
mapping $u(\cd)\mapsto J_0(t,x,y;u(\cd))$,
$$V(t,x,y)\les J_0(t,x,y;u^*(\cd))\les\liminf_{j\to\i}J_0(t,x,y;u_{k_j}(\cd))=V(t,x,y),$$
from which we see $u^*(\cd)$ is an optimal control with respect to $(t,x,y)$.
The uniqueness follows directly from the strict convexity of $u(\cd)\mapsto J_0(t,x,y;u(\cd))$.
\end{proof}

\section{Main results}

Let $Q(\cd)\in L^1(0,T;\dbS^n)$ and $R(\cd)\in L^\i(0,T;\dbS^m)$ be such that
\bel{QR>0}Q(\cd)\ges0, \qq R(\cd)\gg0.\ee
Consider the following Riccati-type equations:
\begin{eqnarray}
&\label{Ric:P}\ \left\{\2n\ba{ll}
\ds \dot P(s)+P(s)A(s)+A(s)^\top P(s)+Q(s)-P(s)B(s)R(s)^{-1}B(s)^\top P(s)=0,\q s\in[0,T),\\
\ns \lim_{s\to T}\min\si(P(s))=\i, \ea\right.&\\
\ns&\label{Ric:Pi}\left\{\2n\ba{ll}
\ds \dot\Pi(s)+\Pi(s)A(s)+A(s)^\top\Pi(s)-Q(s)+\Pi(s)B(s)R(s)^{-1}B(s)^\top\Pi(s)=0,\q s\in(t,T],\\
\ns \lim_{s\to t}\min\si(\Pi(s))=\i, \ea\right.&
\end{eqnarray}
where $\si(M)$ denotes the spectrum of a matrix $M$. Our first result can be stated as follows.

\begin{theorem}\label{thm:main-1}\sl
Let {\rm(H1)--(H2)} hold. Then the Riccati equations \rf{Ric:P} and \rf{Ric:Pi} admit unique solutions
$P(\cd)\in C([0,T);\dbS^n_+)$ and $\Pi(\cd)\in C((t,T];\dbS^n_+)$, respectively.
\end{theorem}

The proof of Theorem \ref{thm:main-1} will be given in the Section 5.
Let us for the moment look at some properties of the solution $P(\cd)$ to \rf{Ric:P}.
Consider the matrix-valued ODE
\bel{Phi}\lt\{\2n\ba{ll}
\ds \dot\F(s)=\big[A(s)-B(s)R(s)^{-1}B(s)^\top P(s)\big]\F(s),\qq s\in[0,T),\\
\ns\ds \F(0)=I. \ea\rt.\ee
Obviously, \rf{Phi} admits a unique solution $\F(\cd)\in C([0,T);\dbR^{n\times n})$
which is invertible. However, one cannot conclude hastily that the solution $\F(\cd)$
could be extended to the whole interval $[0,T]$ because $P(s)$ explodes as $s\uparrow T$.
The following result gives a rigorous discussion of this issue.

\begin{proposition}\label{prop:P-1}\sl
Let {\rm(H1)--(H2)} hold, and let $P(\cd)\in C([0,T);\dbS^n_+)$ be the solution to the Riccati
equation \rf{Ric:P}. The solution $\F(\cd)$ of \rf{Phi} satisfies $\lim_{s\to T}\F(s)=0$.
\end{proposition}

\begin{proof}
Let $x\in\dbR^n$ be arbitrary. For any $0<s<T$, integration by parts gives
\begin{eqnarray*}
&&\lan P(s)\F(s)x,\F(s)x\ran-\lan P(0)x,x\ran\\
&&=\int_0^s\Blan\Big\{\dot P(r)+P(r)\big[A(r)-B(r)R(r)^{-1}B(r)^\top P(r)\big]\\
&&\qq\q~~+\big[A(r)-B(r)R(r)^{-1}B(r)^\top P(r)\big]^\top P(r)\Big\}\F(r)x,\F(r)x\Bran dr\\
&&=-\int_t^s\blan\big[Q(r)+P(r)B(r)R(r)^{-1}B(r)^\top P(r)\big]\F(r)x,\F(r)x\bran dr\les0.
\end{eqnarray*}
Let $\l_s$ denote the minimal eigenvalue of $P(s)$. Then the above yields
$$\l_s|\F(s)x|^2\les\lan P(s)\F(s)x,\F(s)x\ran\les\lan P(0)x,x\ran.$$
Since $\l_s\to\i$ as $s\to T$ and $x$ is arbitrary, we must have $\lim_{s\to T}\F(s)=0$.
\end{proof}

In light of Proposition \ref{prop:P-1}, the solution $\F(\cd)$ of \rf{Phi} has
a continuous extension to $[0,T]$. Thus, the ODE
\bel{Psi}\lt\{\2n\ba{ll}
\ds \dot\Psi(s)=-A(s)^\top\Psi(s)-Q(s)\F(s),\qq s\in[0,T],\\
\ns\ds \Psi(0)=P(0) \ea\rt.\ee
admits a unique solution $\Psi(\cd)$ on the whole interval $[0,T]$, and we have the following:

\begin{proposition}\sl
Let {\rm(H1)--(H2)} hold, and let $P(\cd)\in C([0,T);\dbS^n_+)$ be the solution to the
Riccati equation \rf{Ric:P}. The solution $\F(\cd)$ of \rf{Phi} satisfies
$$\lim_{s\to T}P(s)\F(s)=\Psi(T).$$
\end{proposition}

\begin{proof}
By differentiating we get
\begin{eqnarray*}
{d\over ds}[P(s)\F(s)] \3n&=\3n& \dot P(s)\F(s)+P(s)\dot\F(s)\\
\3n&=\3n& \big[\dot P(s)+P(s)A(s)-P(s)B(s)R(s)^{-1}B(s)^\top P(s)\big]\F(s)\\
\3n&=\3n& -A(s)^\top [P(s)\F(s)]-Q(s)\F(s),\qq s\in[0,T).
\end{eqnarray*}
Thus, $P(\cd)\F(\cd)$ satisfies equation \rf{Psi} on the interval $[0,T)$.
By uniqueness of solutions, we must have $P(s)\F(s)=\Psi(s)$ for all $s\in[0,T)$.
The desired result then follows immediately.
\end{proof}

Let $\G=\{(\l_1,\ldots,\l_k):\l_i\ges0,~i=1,\ldots,k\}$ and define for $\l=(\l_1,\ldots,\l_k)\in\G$,
\bel{QR(lamda)}Q(\l,s)=Q_0(s)+\sum_{i=1}^k \l_iQ_i(s), \q R(\l,s)=R_0(s)+\sum_{i=1}^k \l_iR_i(s).\ee
We have from Theorem \ref{thm:main-1} that under (H1)--(H2), the following ($\l$-dependent) Riccati
equations are uniquely solvable:
\begin{eqnarray}
&&\label{Ric:P(l)}\left\{\2n\ba{ll}
\ds \dot P(\l,s)+P(\l,s)A(s)+A(s)^\top P(\l,s)+Q(\l,s)\\
\ns\ds~~-P(\l,s)B(s)R(\l,s)^{-1}B(s)^\top P(\l,s)=0,\qq s\in[0,T),\\
\ns \lim_{s\to T}\min\si(P(\l,s))=\i,\ea\right.\\
\ns&&\label{Ric:Pi(l)}\left\{\2n\ba{ll}
\ds \dot\Pi(\l,s)+\Pi(\l,s)A(s)+A(s)^\top\Pi(\l,s)-Q(\l,s)\\
\ns\ds~~+\Pi(\l,s)B(s)R(\l,s)^{-1}B(s)^\top\Pi(\l,s)=0,\qq s\in(t,T],\\
\ns \lim_{s\to t}\min\si(\Pi(\l,s))=\i.\ea\right.
\end{eqnarray}
Let $\F(\l,\cd)$ and $\Psi(\l,\cd)$ be the solutions to
\bel{Phi(l)}\lt\{\2n\ba{ll}
\ds \dot\F(\l,s)=\big[A(s)-B(s)R(\l,s)^{-1}B(s)^\top P(\l,s)\big]\F(\l,s),\qq s\in[0,T),\\
\ns\ds \F(\l,0)=I \ea\rt.\ee
and
\bel{Psi(l)}\lt\{\2n\ba{ll}
\ds \dot\Psi(\l,s)=-A(s)^\top\Psi(\l,s)-Q(\l,s)\F(\l,s),\qq s\in[0,T],\\
\ns\ds \Psi(\l,0)=P(\l,0), \ea\rt.\ee
respectively. We are ready for our next main result, whose proof will be given in Section 6.

\begin{theorem}\label{thm:main-2}\sl
Let {\rm(H1)--(H2)} hold, and let $(t,x,y)\in [0,T)\times\dbR^n\times\dbR^n$ be given.
Suppose there exists at least one strictly feasible control w.r.t. $(t,x,y)$.
Then the function $L(\,\cd\,,t,x,y):\G\to\dbR$ defined by
$$L(\l,t,x,y)\deq\lan P(\l,t)x,x\ran-2\lan\Psi(\l,T)\F(\l,t)^{-1}x,y\ran
+\lan \Pi(\l,T)y,y\ran-\l^\top c $$
achieves its maximum at some $\l^*\in\G$, and the optimal control of Problem {\rm(CLQ)} is given by
\bel{u(l*)}u(\l^*,s)=-R(\l^*,s)^{-1}B(s)^\top\big[P(\l^*,s)X(\l^*,s)+\eta(\l^*,s)\big],\qq s\in[t,T),\ee
where
$$\eta(\l^*,s)=-\big[\Psi(\l^*,T)\F(\l^*,s)^{-1}\big]^\top y,\qq s\in[0,T),$$
and $X(\l^*,\cd)$ is the solution to the closed-loop system
$$\left\{\2n\ba{ll}
\ds\dot X(\l^*,s)=\big[A(s)-B(s)R(\l^*,s)^{-1}B(s)^\top P(\l^*,s)\big]X(\l^*,s)\\
\ns\ds\hphantom{\dot X(\l^*,s)=} -B(s)R(\l^*,s)^{-1}B(s)^\top\eta(\l^*,s),\qq s\in[t,T), \\
\ns\ds X(\l^*,t)=x. \ea\right.$$
\end{theorem}

\begin{remark}\rm
Form the representation \rf{u(l*)}, we see that the optimal control of Problem (CLQ)
is not only a feedback of the current state, but also a feedback of the target.
\end{remark}

\section{Approach by standard LQ problems}

In this section we approach Problem (CLQ) by a class of LQ problems without constraints.
Our first step is to reduce Problem (CLQ) to an LQ problem without the integral quadratic
constraints by means of the Lagrangian duality.
It is worth noting that the reduced LQ problem is still not standard because the terminal
state is fixed.

\ms

For $\l\in\G=\{(\l_1,\ldots,\l_k):\l_i\ges0,~i=1,\ldots,k\}$, let
\bel{J(lamda)}\ba{lll}
\ds J(\l,t,x,y;u(\cd)) &\3n=\3n&\ds J_0(t,x,y;u(\cd))+\sum_{i=1}^k\l_i J_i(t,x,y;u(\cd))\\
\ns&\3n=\3n&\ds \int_t^T\[\lan Q(\l,s)X(s),X(s)\ran+\lan R(\l,s)u(s),u(s)\ran\]ds,
\ea\ee
where $Q(\l,s)$ and $R(\l,s)$ are defined by \rf{QR(lamda)}. Consider the following Problem:

\ms

\bf Problem (CLQ*). \rm
For any given initial pair $(t,x)\in[0,T)\times\dbR^n$ and any target $y\in\dbR^n$,
find a $u^*(\l,\cd)\in\cU(t,x,y)$ such that
$$J(\l,t,x,y;u^*(\l,\cd))=\inf_{u(\cd)\in\cU(t,x,y)}J(\l,t,x,y;u(\cd))\deq V(\l,t,x,y).$$

\ss

By the Lagrange duality theorem, we have the following result.

\begin{theorem}\label{thm:duality}\sl
Let {\rm (H1)--(H2)} hold, and let $(t,x,y)\in [0,T)\times\dbR^n\times\dbR^n$ be given.
Then for any $\l\in\G$, Problem {\rm(CLQ*)} admits a unique optimal control $u^*(\l,\cd)$.
If, in addition, there exists a strictly feasible control w.r.t. $(t,x,y)$,
then the dual functional
\bel{dual-f}\f(\l)\deq J(\l,t,x,y;u^*(\l,\cd))-\l^\top c, \qq \l\in\G\ee
achieves its maximum at some $\l^*\in\G$, and the unique optimal control of Problem {\rm(CLQ)}
is $u^*(\l^*,\cd)$.
\end{theorem}

\begin{proof}
The first assertion can be proved by a similar argument used in the proof of Theorem \ref{thm:uni+exi},
and the second assertion follows from the Lagrange duality theorem \cite[Theorem 1, page 224]{Luenberger 1969}.
\end{proof}

Once we find out the optimal control of Problem (CLQ*) and derive the value function $V(\l,t,x,y)$,
we shall be able to calculate the gradient of the dual functional \rf{dual-f} and solve the original
Problem (CLQ). In order to obtain an explicit representation of the optimal control for Problem (CLQ*),
we adopt the penalty approach, in which Problem (CLQ*) is approximated by a sequence of standard LQ
problems where the terminal states are unconstrained.

\ms

Let $Q(\cd)\in L^1(0,T;\dbS^n)$ and $R(\cd)\in L^\i(0,T;\dbS^m)$ be such that \rf{QR>0} holds.
For each $\l\in\G$, the matrices in the cost function \rf{J(lamda)} have the same properties
as $Q(\cd)$ and $R(\cd)$. So in what follows we shall simply consider Problem (CLQ*) with the
cost functional
$$ J(t,x,y;u(\cd)) = \int_t^T\[\lan Q(s)X(s),X(s)\ran+\lan R(s)u(s),u(s)\ran\]ds, $$
and the corresponding value function will be denoted by $V(t,x,y)$.
For every integer $i\ges1$ let
\bel{J_i}J_i(t,x,y;u(\cd))=i|X(T)-y|^2+\int_t^T\[\lan Q(s)X(s),X(s)\ran+\lan R(s)u(s),u(s)\ran\]ds.\ee
The family of standard LQ problems, parameterized by $i$, is defined as follows.

\ms

\bf Problem (LQ)$_i$. \rm For any given $(t,x,y)\in[0,T)\times\dbR^n\times\dbR^n$,
find a $u_i^*(\cd)\in\cU[t,T]$ such that
$$J_i(t,x,y;u_i^*(\cd))=\inf_{u(\cd)\in\cU[t,T]}J_i(t,x,y;u(\cd))\deq V_i(t,x,y).$$

\ss

The solution of the above Problem (LQ)$_i$ can be obtained by using a completion-of-squares
technique via the Riccati equation
\bel{Ric-P_i}\left\{\2n\ba{ll}
\ds\dot P_i(s)+P_i(s) A(s)+A(s)^\top P_i(s)+Q(s)-P_i(s) B(s)R(s)^{-1}B(s)^\top P_i(s)=0,\qq s\in[0,T],\\
\ns\ds P_i(T)=i I,\ea\right.\ee
see, e.g., \cite{Yong-Zhou 1999} for a thorough study of the Riccati approach
(see also \cite{Sun-Li-Yong 2016} for some new developments). More precisely,
let $P_i(\cd)\in C([0,T];\dbS^n)$ be the unique solution of \rf{Ric-P_i},
and let $\eta_i(\cd)\in C([0,T];\dbR^n)$ be the solution of
\bel{eta_i}\left\{\2n\ba{ll}
\ds\dot\eta_i(s)=-\big[A(s)-B(s)R(s)^{-1}B(s)^\top P_i(s)\big]^\top\eta_i(s),\qq s\in[0,T],\\
\ns\ds\eta_i(T)=-i y.\ea\right.\ee
The unique optimal control $u_i^*(\cd)$ of Problem (LQ)$_i$ (for $(t,x,y)$) is given by the
following state feedback form:
\bel{u*_i}u^*_i(s)=-R(s)^{-1}B(s)^\top\big[P_i(s)X^*_i(s)+\eta_i(s)\big],\qq s\in[t,T],\ee
where $X_i^*(\cd)$ is the solution to the {\it closed-loop} system
\bel{X_i*}\left\{\2n\ba{ll}
\ds\dot X^*_i(s)=\big[A(s)-B(s)R(s)^{-1}B(s)^\top P_i(s)\big]X^*_i(s)
-B(s)R(s)^{-1}B(s)^\top\eta_i(s),\qq s\in[t,T],\\
\ns\ds X^*_i(t)=x.\ea\right.\ee
Moreover, the value function of Problem (LQ)$_i$ has the following representation:
$$V_i(t,x,y)=\lan P_i(t)x,x\ran+2\lan\eta_i(t),x\ran+i|y|^2
-\int_t^T\blan R(s)^{-1}B(s)^\top\eta_i(s),B(s)^\top\eta_i(s)\bran ds.$$
In particular, if $y=0$, the solution $\eta_i(\cd)$ of \rf{eta_i} is identically zero, and
$$V_i(t,x,0)=\lan P_i(t)x,x\ran,\qq\forall (t,x)\in[0,T]\times\dbR^n.$$
Because the cost functional is nonnegative and the weight on the square of the terminal state
is positive, it is not difficult to see by contradiction that $P_i(t)>0$ for all $t\in[0,T]$.

\ms

Note that $J_i(t,x,y;u(\cd))$ is nondecreasing in $i$. Hence, when the system $[A,B]$ is
completely controllable, it is expected that the sequence $\{u^*_i(\cd)\}^\i_{i=1}$ defined
by \rf{u*_i} converges to the unique optimal control of Problem (CLQ*) for the initial pair
$(t,x)$ and target $y$. Actually, we have the following result.

\begin{theorem}\label{thm:lim-LQi}\sl
Let {\rm(H1)--(H2)} hold. For $(t,x,y)\in[0,T)\times\dbR^n\times\dbR^n$,
let $(u_i^*(\cd),X^*_i(\cd))$ be the corresponding optimal pair of Problem {\rm(LQ)$_i$}.
We have the following:
\ben[\indent\rm(i)]
\item $V_i(t,x,y)\uparrow V(t,x,y)$ as $i\to\i$.
\item $\{u_i^*(\cd)\}^\i_{i=1}$ has a subsequence converging weakly to the unique optimal
control of Problem {\rm(CLQ*)} with respect to $(t,x,y)$.
\een\end{theorem}

\begin{proof}
We have seen in Theorem \ref{thm:duality} that Problem (CLQ*) is uniquely solvable.
Let $u^*(\cd)\in\cU(t,x,y)$ be the unique optimal control of Problem (CLQ*) with respect to $(t,x,y)$,
and let $X^*(\cd)$ be the corresponding optimal trajectory.
Since $Q(\cd),R(\cd)\ges0$ and $X^*(T)=y$, we have
\begin{eqnarray}
i|X_i^*(T)-y|^2 \4n&\les&\4n J_i(t,x,y;u_i^*(\cd))=V_i(t,x,y),\nonumber\\
\label{V_i<V} V_i(t,x,y) \4n&\les&\4n J_i(t,x,y;u^*(\cd))=J(t,x,y;u^*(\cd))=V(t,x,y),
\end{eqnarray}
from which we conclude that
$$\lim_{i\to\i}X_i^*(T)=y.$$
On the other hand, since $Q(\cd)\ges0$ and $R(\cd)\gg0$, there exists a $\d>0$ such that
$$J_i(t,x,y;u(\cd))\ges\d\int_t^T|u(s)|^2ds,\qq\forall\, u(\cd)\in L^2(t,T;\dbR^m),$$
which, together with \rf{V_i<V}, yields
$$\int_t^T|u_i^*(s)|^2ds\les \d^{-1}J_i(t,x,y;u_i^*(s))=\d^{-1}V_i(t,x,y)
\les\d^{-1}V(t,x,y)<\i,\qq\forall\,i\ges1.$$
Thus, $\{u_i^*(\cd)\}^\i_{i=1}$ is bounded in the Hilbert space $L^2(t,T;\dbR^m)$ and hence
admits a weakly convergent subsequence $\{u_{i_k}^*(\cd)\}_{k=1}^\i$. Let $v(\cd)$ be the weak
limit of $\{u_{i_k}^*(\cd)\}_{k=1}^\i$. The sequential weak lower semicontinuity of the mapping
$u(\cd)\mapsto J(t,x,y;u(\cd))$ gives
\bel{J(v)<V}\ba{lll}
\ds J(t,x,y;v(\cd))\4n&\les&\4n\ds \liminf_{k\to\i}J(t,x,y;u_{i_k}^*(\cd))
\les\liminf_{k\to\i}J_{i_k}(t,x,y;u_{i_k}^*(\cd))\\
\ns\4n&=&\4n\ds \lim_{k\to\i}V_{i_k}(t,x,y)\les V(t,x,y).\ea\ee
The above inequality will imply that $v(\cd)$ coincides with the unique optimal control
$u^*(\cd)$ of Problem (CLQ*) with respect to $(t,x,y)$ once we prove $v(\cd)\in\cU(t,x,y)$.
Define a continuous, convex mapping $\sL:L^2(t,T;\dbR^m)\to \dbR^n$ by the following:
$$\sL(u(\cd))=X(T;t,x,u(\cd)),$$
where $X(\cd\,;t,x,u(\cd))$ is the solution to the state equation \rf{state} corresponding
to $u(\cd)$ and $(t,x)$. By Mazur's lemma, one can find $\a_{kj}\in[0,1],~j=1,2\cdots,N_k$
with $\sum_{j=1}^{N_k}\a_{kj}=1$ such that $\sum_{j=1}^{N_k}\a_{kj}u_{i_{k+j}}^*(\cd)$
converges strongly to $v(\cd)$ as $k\to\i$. Thus,
\begin{eqnarray*}
X(T;t,x,v(\cd)) \4n&=&\4n \sL(v(\cd))=\lim_{k\to\i}\sL\lt(\sum_{j=1}^{N_k}\a_{kj}u_{i_{k+j}}^*(\cd)\rt)\\
\4n&=&\4n \lim_{k\to\i}\sum_{j=1}^{N_k}\a_{kj}\sL(u_{i_{k+j}}^*(\cd))
=\lim_{k\to\i}\sum_{j=1}^{N_k}\a_{kj}X_{i_{k+j}}^*(T)=y.
\end{eqnarray*}
This shows $v(\cd)\in\cU(t,x,y)$, and hence (ii) holds. Now \rf{J(v)<V} yields
$$V(t,x,y)=J(t,x,y;v(\cd))\les\lim_{k\to\i}V_{i_k}(t,x,y)\les V(t,x,y),$$
and (i) follows readily.
\end{proof}

\section{Riccati equation}

The aim of this section is to investigate the existence and uniqueness of solutions to
the Riccati equations \rf{Ric:P} and \rf{Ric:Pi}. We will focus mainly on \rf{Ric:P} as
the well-posedness of the Riccati equation \rf{Ric:Pi} can be obtained by a simple
time-reversal on the result for \rf{Ric:P}.

\ms

First, we present the following result concerning the uniqueness of solutions to the
Riccati equation \rf{Ric:P}.

\begin{proposition}\sl
Let {\rm(H1)} hold. Then the Riccati equation \rf{Ric:P} has at most one solution
$P(\cd)\in C([0,T);\dbS^n)$.
\end{proposition}

\begin{proof}
Suppose that $P_1(\cd),P_2(\cd)\in C([0,T);\dbS^n)$ are two solutions of \rf{Ric:P}.
Take $\t\in[0,T)$ such that $P_1(s),P_2(s)>0$ on $[\t,T)$, and set for $i=1,2$,
$$\Si_i(s)=\left\{\2n\ba{ll}
\ds P_i(s)^{-1},&~ s\in[\t,T),\\
\ns\ds 0,&~ s=T.\ea\right.$$
By evaluating ${d\over ds}[P_i(s)\Si_i(s)]=0$, we see that both $\Si_1(\cd)$ and $\Si_2(\cd)$
solve the following ODE:
$$\left\{\2n\ba{ll}
\ds\dot\Si-A\Si-\Si A^\top-\Si Q\Si+BR^{-1}B^\top=0,\qq s\in[0,T],\\
\ns\ds \Si(T)=0.\ea\right.$$
Thus, $\Pi(\cd)\deq\Si_1(\cd)-\Si_2(\cd)$ satisfies $\Pi(T)=0$ and
\begin{eqnarray*}
\dot\Pi&\3n=&\3n A\Pi+\Pi A^\top+\Si_1 Q\Si_1-\Si_2 Q\Si_2\\
&\3n=&\3n A\Pi+\Pi A^\top+\Pi Q\Si_1+\Si_2 Q\Pi\\
&\3n=&\3n (A+\Si_2 Q)\Pi+\Pi (A^\top+Q\Si_1)
\end{eqnarray*}
on $[\t,T]$. By a standard argument using Gronwall's inequality we obtain $\Pi(s)=0$ for all
$s\in[\t,T]$. This shows $P_1(\cd)=P_2(\cd)$ on $[\t,T]$. Now let $\G(\cd)=P_1(\cd)-P_2(\cd)$.
Then $\G(\t)=0$ and
\begin{eqnarray*}
0&\3n=&\3n \dot{\G}+\G A+A^\top\G-P_1BR^{-1}B^\top P_1+P_2BR^{-1}B^\top P_2\\
&\3n=&\3n \dot{\G}+\G A+A^\top\G-\G BR^{-1}B^\top P_1-P_2BR^{-1}B^\top\G\\
&\3n=&\3n \dot{\G}+\G (A-BR^{-1}B^\top P_1)+(A^\top-P_2BR^{-1}B^\top)\G
\end{eqnarray*}
on $[0,\t]$. Again by Gronwall's inequality we obtain $P_1(\cd)=P_2(\cd)$ on $[0,\t]$.
\end{proof}

Next we prove the existence of solutions to the Riccati equation \rf{Ric:P}.
The basic idea is to pass to the limit in \rf{Ric-P_i}. Theorem \ref{thm:lim-LQi}
will guarantee the existence of the limit $P(s)\equiv\lim_{i\to\i}P_i(s)$,
which is a solution of \rf{Ric:P}.

\begin{theorem}\label{thm:exi-Ric}\sl
Let {\rm(H1)--(H2)} hold. Then the Riccati equation \rf{Ric:P}
admits a unique solution $P(\cd)\in C([0,T);\dbS_+^n)$. Moreover,
\bel{V(t,x,0)}V(t,x,0)\deq\inf_{u(\cd)\in\cU(t,x,0)}J(t,x,0;u(\cd))=\lan P(t)x,x\ran,
\qq\forall\, (t,x)\in[0,T)\times\dbR^n.\ee
\end{theorem}

\begin{proof}
Consider Problem (LQ)$_i$ with $y=0$. For $i\ges1$, let $P_i(\cd)\in C([0,T];\dbS_+^n)$
be the solution to \rf{Ric-P_i}. Note that in the case of $y=0$, the solution $\eta_i(\cd)$
of \rf{eta_i} is identically zero, and the value function of Problem (LQ)$_i$ is given by
$$V_i(t,x,0)=\lan P_i(t)x,x\ran,\qq (t,x)\in[0,T]\times\dbR^n.$$
Then from Theorem \ref{thm:lim-LQi} (i), we see that for any $t\in[0,T)$, $\{P_i(t)\}^\i_{i=1}$ is an
increasing, bounded sequence, and hence has a limit $P(t)\in\dbS_+^n$ having the property \rf{V(t,x,0)}.
On the other hand, one can easily verify that the control defined by
$$v(s)=-\big[\F_A(s)^{-1}B(s)\big]^\top\lt(\int_t^T\F_A(r)^{-1}B(r)\big[\F_A(r)^{-1}B(r)\big]^\top dr\rt)^{-1}
\F_A(t)^{-1}x\equiv\dbV(s,t)x,\q s\in[t,T]$$
is in $\cU(t,x,0)$, where $\F_A(\cd)$ is the solution of \rf{F_A}.
Thus, with $\dbX(\cd\,,t)$ denoting the solution to the matrix-valued ODE
$$\lt\{\2n\ba{ll}
\ds\dot \dbX(s,t)=A(s)\dbX(s,t)+B(s)\dbV(s,t),\qq s\in[t,T],\\
\ns\ds \dbX(t,t)=I, \ea\rt.$$
we have $X(\cd\,;t,x,v(\cd))=\dbX(\cd\,,t)x$, and hence
\begin{eqnarray*}
\lan P_i(t)x,x\ran
\3n&=\3n& V_i(t,x,0)\les V(t,x,0)\les J(t,x,0;v(\cd))\\
\3n&=\3n& \int_t^T\[\lan Q(s)\dbX(s,t)x,\dbX(s,t)x\ran+\lan R(s)\dbV(s,t)x,\dbV(s,t)x\ran\]ds\\
\3n&\equiv\3n& \lan M(t)x,x\ran,\qq\forall\, t\in[0,T),~\forall\, x\in\dbR^n.
\end{eqnarray*}
Noting that $\dbX(s,t)$ and $\dbV(s,t)$ are continuous functions of $(s,t)$,
we conclude that the function $M(\cd)$ is continuous in $[0,T)$. Hence, $\{P_i(t)\}^\i_{i=1}$
is uniformly bounded on compact subintervals of $[0,T)$, and by the dominated convergence theorem,
we have for any $t\in[0,T)$,
\begin{eqnarray*}
P(t)\3n&=\3n& \lim_{i\to\i}P_i(t)
=\lim_{i\to\i}\lt[P_i(0)-\int_0^t\(P_i A+A^\top P_i+Q-P_i BR^{-1}B^\top P_i\)ds\rt]\\
\3n&=\3n& P(0)-\int_0^t\(P A+A^\top P+Q-P BR^{-1}B^\top P\)ds.
\end{eqnarray*}
This implies that $P(\cd)$ satisfies the differential equation in \rf{Ric:P}.
Finally, since $P(t)\ges P_i(t)$ for all $i\ges1$ and all $t\in[0,T)$, we have
$$\lim_{t\to T}P(t)\ges\lim_{i\to\i}\lim_{t\to T}P_i(t)=\lim_{i\to\i}i I=\i.$$
The proof is completed.
\end{proof}

\begin{remark}\label{rmk:P_i}\rm
From the proof of Theorem \ref{thm:exi-Ric}, we have the following facts:
\begin{enumerate}[\indent\rm(i)]
\item The solution $P_i(t)$ of the Riccati equation \rf{Ric-P_i} is increasing in $i$
and converges to $P(t)$, the solution of the Riccati equation \rf{Ric:P}, for all $t\in[0,T)$
as $i\to\i$.
\item The sequence $\{P_i(t)\}^\i_{i=1}$ is uniformly bounded on compact subintervals of $[0,T)$.
\end{enumerate}
\end{remark}

To show the unique solvability of the Riccati equation \rf{Ric:P},
let us fix $t\in[0,T)$ and define for $t\les s\les T$,
$$\ba{lll}
\ds \bar A(s)=-A(T+t-s), \q& \bar B(s)=-B(T+t-s),\\
\ns \bar Q(s)=Q(T+t-s),  \q& \bar R(s)=R(T+t-s).
\ea$$
For $t\les r<T$, consider the controlled ODE
$$\lt\{\2n\ba{ll}
\ds \dot{\bar X}(s)=\bar A(s)\bar X(s)+\bar B(s)v(s),\qq s\in[r,T],\\
\ns\ds \bar X(r)=y, \ea\rt.$$
and the cost functional
$$\bar J(r,y,x;v(\cd))\deq\int_r^T\[\blan\bar Q(s)\bar X(s),\bar X(s)\bran+\blan\bar R(s)v(s),v(s)\bran\]ds.$$
Using the criterion \rf{regular}, it is not hard to show that system $[\bar A,\bar B]$
is completely controllable. Since $\bar Q(\cd)\ges0$ and $\bar R(\cd)\gg0$ on $[t,T]$,
we have by Theorem \ref{thm:exi-Ric} that the Riccati equation
$$\left\{\2n\ba{ll}
\ds \dot\Si(s)+\Si(s)\bar A(s)+\bar A(s)^\top\Si(s)+\bar Q(s)
-\Si(s)\bar B(s)\bar R(s)^{-1}\bar B(s)^\top\Si(s)=0,\qq s\in[t,T),\\
\ns \lim_{s\to T}\min\si(\Si(s))=\i \ea\right.$$
admits a unique solution $\Si(\cd)\in C([t,T);\dbS_+^n)$. For initial pair $(t,y)$ and target $x=0$,
let $v^*(\cd)$ be the corresponding optimal control of the above problem.
By Theorem \ref{thm:exi-Ric}, the corresponding value is
$$\bar V(t,y,0)\deq\inf_{v(\cd)\in\cU(t,y,0)}\bar J(t,y,0;v(\cd))=\lan\Si(t)y,y\ran.$$
By reversing time,
$$\t=T+t-s, \qq s\in[t,T],$$
we see that
$$u^*(s)\deq v^*(T+t-s),\qq s\in[t,T]$$
is the unique optimal control of Problem (CLQ*) for the initial pair $(t,0)$ and target $y$,
and that $\Pi(s)=\Si(T+t-s)$ is the unique solution to the Riccati equation \rf{Ric:Pi}.
This gives us the following result.

\begin{proposition}\label{prop:exi-Pi}\sl
Let {\rm(H1)--(H2)} hold. Then for any $t\in[0,T)$, the Riccati equation \rf{Ric:Pi}
admits a unique solution $\Pi(\cd)\in C((t,T];\dbS_+^n)$. Moreover,
$$V(t,0,y)\deq\inf_{u(\cd)\in\cU(t,0,y)}J(t,0,y;u(\cd))=\lan\Pi(T)y,y\ran,\qq \forall y\in\dbR^n.$$
\end{proposition}

\begin{proof}[\indent\textbf{Proof of Theorem {\rm\ref{thm:main-1}}}]
The proof follows directly from a combination of Theorem \ref{thm:exi-Ric}
and Proposition \ref{prop:exi-Pi}.
\end{proof}

%\begin{lemma}\sl The system $[\bar A,\bar B]$ is completely controllable on any
%subinterval of $[t,T]$.
%\end{lemma}
%
%\begin{proof} Fix $t\les r_0<r_1\les T$. Let $\F_A(\cd)$ be the solution to ODE \rf{16Aug4-16:40}
%and let $\bar\F(\cd)$ be the solution to
%%
%$$\lt\{\2n\ba{ll}
%%
%\ds \dot{\bar\F}(s)=\bar A(s)\bar\F(s), \qq s\in[t,T],\\
%%
%\ns\ds \bar\F(t)=I. \ea\rt.$$
%%
%It is easy to check that
%%
%$$\bar\F(s)=\F_A(T+t-s)\F_A(T)^{-1},\qq s\in[t,T].$$
%%
%Thus,
%%
%$$\eta^\top\bar\F(s)^{-1}\bar B(s)=0 \qq\ae~s\in[r_0,r_1]$$
%%
%is equivalent to
%%
%$$\eta^\top\F_A(T)\F_A(s)^{-1}B(s)=0 \qq\ae~s\in[T+t-r_1,T+t-r_0].$$
%%
%which implies $\eta=0$ since system $[A,B]$ is completely controllable. The proof is completed.
%\end{proof}

\section{Proof of Theorem \ref{thm:main-2}}

In this section we prove the second main result of the paper, Theorems \ref{thm:main-2}.
Our proof requires some technical lemmas, which we establish first.

\begin{lemma}\label{lmm:1}\sl
Let $1<p<\i$ and let functions $f_n\in L^p$ converge almost everywhere
(or in measure) to a function $f$. Then, a necessary and sufficient condition for convergence of
$\{f_n\}$ to $f$ in the weak topology of $L^p$ is the boundedness of $\{f_n\}$ in the norm of $L^p$.
\end{lemma}

\begin{proof}
The proof can be found in \cite[page 282]{Bogachev 2007}.
\end{proof}

For arbitrary functions $Q(\cd)\ges0$ in $L^1(0,T;\dbS^n)$ and $R(\cd)\gg0$ in $L^\i(0,T;\dbS^m)$,
let $P(\cd)$ be the corresponding solution of the Riccati equation \rf{Ric:P}, and let $\F(\cd)$
and $\Psi(\cd)$ be the solutions to equations \rf{Phi} and \rf{Psi}, respectively.
Recall from Remark \ref{rmk:P_i} that the solution $P_i(\cd)$ of \rf{Ric-P_i} converges to $P(\cd)$
on $[0,T)$ as $i\to\i$. We have the following two lemmas.

\begin{lemma}\label{lmm:2}\sl For $i=1,2,\ldots,$ let $\F_i(\cd)$ be the solution to
$$\left\{\2n\ba{ll}
\ds\dot\F_i(s)=\big[A(s)-B(s)R(s)^{-1}B(s)^\top P_i(s)\big]\F_i(s),\qq s\in[0,T], \\
\ns\ds\F_i(0)=I.\ea\right.$$
We have the following:

\ms

{\rm(i)} $\{\F_i(s)\}^\i_{i=1}$ is uniformly bounded on $[0,T]$, and
$$\lim_{i\to\i}\F_i(s)=\F(s),\qq\forall s\in[0,T].$$

{\rm(ii)} $\{\F_i(s)^{-1}\}^\i_{i=1}$ is uniformly bounded on compact subintervals of $[0,T)$.
\end{lemma}

\begin{proof}
(i) Let $A_i(s)=A(s)-B(s)R(s)^{-1}B(s)^\top P_i(s)$. By the integration by parts formula,
we have for any $s\in[0,T]$,
\bel{F_iP_iF_i-P_i<0}\ba{ll}
\ds\F_i(s)^\top P_i(s)\F_i(s)-P_i(0)\\
\ns\ds=\int_0^s\F_i(r)^\top\big[A_i(r)^\top P_i(r)+\dot P_i(r)+P_i(r)A_i(r)\big]\F_i(r)dr\\
\ns\ds=-\int_0^s\F_i(r)^\top\big[Q(r)+P_i(r)B(r)R(r)^{-1}B(r)^\top P_i(r)\big]\F_i(r) dr\les0.
\ea\ee
Since for any $i\ges1$, $P_i(s)\ges P_1(s)>0$ for all $s\in[0,T]$ and $P(0)\ges P_i(0)$
(see Remark \ref{rmk:P_i} (i)), there exists a constant $\mu>0$ such that
$$\mu\F_i(s)^\top\F_i(s)\les\F_i(s)^\top P_1(s)\F_i(s)
\les\F_i(s)^\top P_i(s)\F_i(s)\les P_i(0)\les P(0).$$
This implies that $|\F_i(s)|^2\les\mu^{-1}\sqrt{n}|P(0)|$ for all $i\ges1$ and all $s\in[0,T]$.
The first assertion follows readily. For the second, denote
$$\Pi(s) = B(s)R(s)^{-1}B(s)^\top P(s),\qq \Pi_i(s) = B(s)R(s)^{-1}B(s)^\top P_i(s),$$
and note that for $s\in[0,T)$,
$$\F_i(s)-\F(s)=\int_0^s\Big\{A_i(r)\big[\F_i(r)-\F(r)\big]+\big[\Pi(r)-\Pi_i(r)\big]\F(r)\Big\}dr.$$
By the Gronwall inequality, we have
$$\big|\F_i(s)-\F(s)\big|\les\int_0^se^{\int_r^s |A_i(u)|du}|\Pi(r)-\Pi_i(r)||\F(r)|dr,\qq s\in[0,T).$$
Since $P_i(s)\to P(s)$ on $[0,T)$ and $\{P_i(s)\}^\i_{i=1}$ is uniformly bounded on compact
subintervals of $[0,T)$ (see Remark \ref{rmk:P_i} (ii)), the dominated convergence theorem yields
$$\lim_{i\to\i}\F_i(s)=\F(s),\qq\forall s\in[0,T).$$
For the case $s=T$, \rf{F_iP_iF_i-P_i<0} gives
$$i\F_i(T)^\top\F_i(T)=\F_i(T)^\top P_i(T)\F_i(T)\les P_i(0)\les P(0),\qq\forall i\ges1,$$
from which follows
$$\lim_{i\to\i}\F_i(T)=0=\F(T).$$

\ss

(ii) One has
$$\left\{\2n\ba{ll}
\ds{d\over ds}\big[\F_i(s)^{-1}\big]=-\F_i(s)^{-1}A_i(s),\qq s\in[0,T], \\
\ns\ds\F_i(0)^{-1}=I.\ea\right.$$
Thus,
$$|\F_i(s)^{-1}|\les|I|+\int_0^s|A_i(r)||\F_i(r)^{-1}|dr,$$
and by the Gronwall inequality we have
$$|\F_i(s)^{-1}|\les|I|e^{\int_0^s|A_i(r)|dr}=\sqrt{n}\exp\lt\{\int_0^s
\Big|A(r)-B(r)R(r)^{-1}B(r)^\top P_i(r)\Big|dr\rt\}.$$
The result then follows immediately form the uniform boundedness of $\{P_i(s)\}^\i_{i=1}$
on compact subintervals of $[0,T)$.
\end{proof}

\begin{lemma}\label{lmm:3}\sl
For $i=1,2,\ldots,$ let $\eta_i(\cd)$ be the solution to \rf{eta_i}.
Then $\{\eta_i(s)\}^\i_{i=1}$ is uniformly bounded on compact subintervals of $[0,T)$, and
\bel{lim-eta_i}\lim_{i\to\i}\eta_i(s)=-\big[\Psi(T)\F(s)^{-1}\big]^\top y,\qq \forall s\in[0,T).\ee
\end{lemma}

\begin{proof}
It is easy to verify that
\bel{eta_i=}\eta_i(s)=-i\big[\F_i(T)\F_i(s)^{-1}\big]^\top y,\qq s\in[0,T].\ee
By Lemma \ref{lmm:2}, $\lim_{i\to\i}\F_i(s)=\F(s)$ for all $s\in[0,T]$.
So in order to prove \rf{lim-eta_i}, it remains to show
\bel{lim-iF_i(T)}\lim_{i\to\i}i\F_i(T)=\Psi(T).\ee
For this, let $\Psi_i(s)=P_i(s)\F_i(s)$. By differentiating we get
\begin{eqnarray*}
\dot\Psi_i(s)\3n&=\3n& \dot P_i(s)\F_i(s)+P_i(s)\dot\F_i(s)\\
\3n&=\3n& \big[\dot P_i(s)+P_i(s)A(s)-P_i(s)B(s)R(s)^{-1}B(s)^\top P_i(s)\big]\F_i(s)\\
\3n&=\3n& -A(s)^\top P_i(s)\F_i(s)-Q(s)\F_i(s)\\
\3n&=\3n& -A(s)^\top\Psi_i(s)-Q(s)\F_i(s).
\end{eqnarray*}
Thus, $P_i(\cd)\F_i(\cd)$ solves the following ODE:
$$\left\{\2n\ba{ll}
\ds\dot\Psi_i(s)=-A(s)^\top\Psi_i(s)-Q(s)\F_i(s),\qq s\in[0,T], \\
\ns\ds\Psi_i(0)=P_i(0).\ea\right.$$
Since $P_i(0)\to P(0)$, $\F_i(s)\to \F(s)$ as $i\to\i$ and $\{\F_i(s)\}^\i_{i=1}$ is
uniformly bounded on $[0,T]$, we conclude by the Gronwall inequality that
$$\lim_{i\to\i}\Psi_i(s)=\Psi(s),\qq\forall s\in[0,T].$$
In particular,
$$\lim_{i\to\i}i\F_i(T)=\lim_{i\to\i}P_i(T)\F_i(T)=\lim_{i\to\i}\Psi_i(T)=\Psi(T).$$
Finally, the uniform boundedness of $\{\eta_i(s)\}^\i_{i=1}$ on compact subintervals of $[0,T)$
follows from \rf{eta_i=}, \rf{lim-iF_i(T)}, and that of $\{\F_i(s)^{-1}\}^\i_{i=1}$.
\end{proof}

\begin{proof}[\indent\textbf{Proof of Theorem {\rm\ref{thm:main-2}}}]
For arbitrary but fixed $\l\in\G=\{(\l_1,\ldots,\l_k):\l_i\ges0,~i=1,\ldots,k\}$, denote
$$ Q(s)\equiv Q(\l,s)=Q_0(s)+\sum_{i=1}^k \l_iQ_i(s),
\q R(s)\equiv R(\l,s)=R_0(s)+\sum_{i=1}^k \l_iR_i(s).$$
Let $P(\cd)\equiv P(\l,\cd)$, $\Pi(\cd)\equiv\Pi(\l,\cd)$, $\F(\cd)\equiv\F(\l,\cd)$,
and $\Psi(\cd)\equiv\Psi(\l,\cd)$ be the solutions to \rf{Ric:P(l)}, \rf{Ric:Pi(l)},
\rf{Phi(l)}, and \rf{Psi(l)}, respectively.
According to Theorem \ref{thm:duality}, it suffices to show ¡¡
\bel{suf-1}\ba{lll}
\ds V(\l,t,x,y)\3n&\deq&\3n\ds \inf_{u(\cd)\in\cU(t,x,y)}J(\l,t,x,y;u(\cd))\\
\ns\3n&=&\3n\ds\lan P(t)x,x\ran-2\lan\Psi(T)\F(t)^{-1}x,y\ran+\lan \Pi(T)y,y\ran, \ea\ee
and that the (unique) optimal control of Problem (CLQ*) with the cost functional
$J(\l,t,x,y;u(\cd))$ is given by
\bel{suf-2}u^*(s)=-R(s)^{-1}B(s)^\top\big[P(s)X^*(s)+\eta(s)\big],\qq s\in[t,T),\ee
where
$$\eta(s)=-\big[\Psi(T)\F(s)^{-1}\big]^\top y,\qq s\in[0,T),$$
and $X^*(\cd)$ is the solution to
$$\left\{\2n\ba{ll}
\ds\dot X^*(s)=\big[A(s)-B(s)R(s)^{-1}B(s)^\top P(s)\big]X^*(s)-B(s)R(s)^{-1}B(s)^\top\eta(s),\qq s\in[t,T), \\
\ns\ds X^*(t)=x. \ea\right.$$
For this we use Theorem \ref{thm:lim-LQi}. Recall from Section 4 that the value function
of the corresponding Problem (LQ)$_i$ is
$$V_i(t,x,y)=\lan P_i(t)x,x\ran+2\lan\eta_i(t),x\ran+V_i(t,0,y)$$
and converges pointwise to $V(\l,t,x,y)$. Letting $i\to\i$, we obtain \rf{suf-1} from
Remark \ref{rmk:P_i} (i), Lemma \ref{lmm:3}, and Proposition \ref{prop:exi-Pi}.
To prove \rf{suf-2}, let $X_i^*(\cd)$ be the solutions to \rf{X_i*} and set
$$\Pi(s)=B(s)R(s)^{-1}B(s)^\top, \qq A_i(s)=A(s)-B(s)R(s)^{-1}B(s)^\top P_i(s).$$
Then we have for any $t\les s<T$,
$$X^*_i(s)-X^*(s)=\int_t^s\Big\{A_i(r)[X^*_i(r)-X^*(r)]
+\Pi(r)[P(r)-P_i(r)]X^*(r)+\Pi(r)[\eta(r)-\eta_i(r)]\Big\}dr.$$
An application of the Gronwall inequality yields
$$|X^*_i(s)-X^*(s)|\les\int_t^s e^{\int_r^s |A_i(u)|du}|\Pi(r)|\Big\{|P(r)-P_i(r)||X^*(r)|
+|\eta(r)-\eta_i(r)|\Big\}dr,\q\forall t\les s<T.$$
Since
$$\lim_{i\to\i}P_i(s)=P(s),\q\lim_{i\to\i}\eta_i(s)=\eta(s),\qq \forall s\in[0,T),$$
and the sequences $\{P_i(s)\}^\i_{i=1}$ and $\{\eta_i(s)\}^\i_{i=1}$ are uniformly bounded
on compact subintervals of $[0,T)$ (see Remark \ref{rmk:P_i} (ii) and Lemma \ref{lmm:3}),
we have by the dominated convergence theorem,
$$\lim_{i\to\i}X^*_i(s)=X^*(s),\qq \forall s\in[0,T).$$
It follows that the sequence $\{u^*_i(\cd)\}^\i_{i=1}$ defined by \rf{u*_i} converges
to $u^*(s)$ for all $s\in[t,T)$ as $i\to\i$. On the other hand, from the proof of
Theorem \ref{thm:lim-LQi} we see that $\{u^*_i(\cd)\}^\i_{i=1}$ is bounded in the norm
of $L^2(t,T;\dbR^m)$. Thus, by Lemma \ref{lmm:1}, $\{u^*_i(\cd)\}^\i_{i=1}$ converges
weakly to $u^*(\cd)$ in $L^2(t,T;\dbR^m)$. The desired result then follows from Theorem
\ref{thm:lim-LQi} (ii).
\end{proof}

\section{Examples}

In this section we present two examples illustrating the results obtained.
In the first example, the integral quadratic constraints are absent,
in which case the optimal parameter $\l^*$ in Theorem \ref{thm:main-2} is obviously zero.
Such kind of problems might represent the selection of a thrust program for a aircraft
which must reach the destination limits in a given time.

\begin{example}\rm Consider the one-dimensional state equation
$$\left\{\2n\ba{ll}
\ds \dot X(s)=X(s)+u(s),\q s\in[0,T],\\
\ns X(0)=x,\ea\right.$$
and the cost functional
$$J(x,y;u(\cd))=\int_0^T |u(s)|^2 ds.$$
Given the initial state $x$ and the target $y$, we seek the control $u^*(\cd)\in L^2(0,T;\dbR)$
minimizing $J(x,y;u(\cd))$, while satisfying the terminal constraint
$$X^*(T)\equiv X(T;x,u^*(\cd))=y.$$
So ${1\over2}J(x,y;u^*(\cd))$ gives the least control energy needed to reach the target $y$
at time $T$ from the initial state $x$.

\ms

We now apply Theorem \ref{thm:main-2} to find the optimal control $u^*(\cd)$.
As mentioned at the beginning of this section, the optimal parameter is zero.
Thus the corresponding Riccati equations become
\begin{eqnarray*}
&&\left\{\2n\ba{ll}
\ds \dot P(s)+2P(s)-P(s)^2=0, \q s\in[0,T),\\
\ns \lim_{s\to T}P(s)=\i,\ea\right.\\
\ns&&\left\{\2n\ba{ll}
\ds \dot\Pi(s)+2\Pi(s)+\Pi(s)^2=0, \q s\in(0,T],\\
\ns \lim_{s\to 0}\Pi(s)=\i,\ea\right.
\end{eqnarray*}
and the corresponding ODEs become
\begin{eqnarray*}
&&\left\{\2n\ba{ll}
\ds \dot \F(s)=[1-P(s)]\F(s), \q s\in[0,T),\\
\ns \F(0)=1, \ea\right.\\
\ns&&\left\{\2n\ba{ll}
\ds \dot \Psi(s)=-\Psi(s), \q s\in[0,T],\\
\ns \Psi(0)=P(0). \ea\right.
\end{eqnarray*}
A straightforward calculation leads to
$$P(s)={2\over 1-e^{2(s-T)}},\q s\in[0,T);\qq \Pi(s)={2\over e^{2s}-1},\q s\in(0,T],$$
and by the variation of constants formula,
$$\F(s)={e^{2T-s}-e^s\over e^{2T}-1},\q \Psi(s)={2e^{2T-s}\over e^{2T}-1},\q s\in[0,T].$$
Now the closed-loop system reads
$$\left\{\2n\ba{ll}
\ds\dot X^*(s)=[1-P(s)]X^*(s)-\eta(s),\q s\in[0,T), \\
\ns\ds X^*(0)=x,\ea\right.$$
where
$$\eta(s)=-\big[\Psi(T)\F(s)^{-1}\big]^\top y=-{2e^Ty\over e^{2T-s}-e^s},\q s\in[0,T).$$
A bit of computation using the variation of constants formula shows that
$$X^*(s)={e^{2T-s}-e^s\over e^{2T}-1}\,x
+{(e^s-e^{-s})e^T\over e^{2T}-1}\,y,\q s\in[0,T].$$
Thus, the optimal control $u^*(\cd)$ is given by
$$u^*(s)= -[P(s)X^*(s)+\eta(s)]={2e^{T-s}\over 1-e^{2T}}\big(e^Tx-y\big),\q s\in[0,T],$$
and the least control energy needed to reach the target $y$ at time $T$ from the initial
state $x$ is given by
\begin{eqnarray*}
{1\over2}J(x,y;u^*(\cd))
\3n&=&\3n {1\over2}\[\lan P(0)x,x\ran-2\lan\Psi(T)\F(0)^{-1}x,y\ran+\lan \Pi(T)y,y\ran\],\\
\3n&=&\3n {1\over e^{2T}-1}\big(e^Tx-y\big)^2.
\end{eqnarray*}
\end{example}

Now we present an example in which the control energy is limited.
Such kind of problems may arise when minimizing flight cost of completing
the trip in a given time with finite fuel.

\begin{example}\rm
Consider the one-dimensional state equation
$$\left\{\2n\ba{ll}
\ds \dot X(s) = X(s)+u(s), \q s\in[0,1],\\
\ns X(0)=1.\ea\right.$$
We want to minimize
$$ J_0(u(\cd)) = \int_0^1\[15|X(s)|^2+|u(s)|^2\]ds $$
over all controls $u(\cd)\in L^2(0,1;\dbR)$ subject to
$$ X(1)=0, \qq J_1(u(\cd))\equiv\int_0^1|u(s)|^2ds\les 3. $$
To this end, we note that in this example the equation for $P(\l,\cd)$ ($\l\ges0$) becomes
$$\left\{\2n\ba{ll}
\ds\dot P(\l,s)+2P(\l,s)+15-{P(\l,s)^2\over 1+\l}=0,\q s\in[0,1),\\
\ns\lim_{s\to 1}P(\l,s)=\i.\ea\right.$$
It is easily verified that
$$P(\l,s) = \l+1+\sqrt{(\l+1)(\l+16)} + {2\sqrt{(\l+1)(\l+16)}\,\G(\l,s)\over\G(\l,1)-\G(\l,s)},\q s\in[0,1),$$
where
$$\G(\l,s)=e^{{2\sqrt{(\l+1)(\l+16)}\over \l+1}s}.$$
By calculating the derivative of
$$L(\l)\deq P(\l,0)-3\l,$$
we obtain the optimal parameter $\l^*\approx 0.1869$.
Now the closed-loop system reads
$$\left\{\2n\ba{ll}
\ds\dot X^*(s)=\lt[1-{P(\l^*,s)\over 1+\l^*}\rt]X^*(s),\q s\in[0,1), \\
\ns\ds X^*(0)=1.\ea\right.$$
By the variation of constants formula we have
$$X^*(s)= {\G(\l^*,1)-\G(\l^*,s)\over (\G(\l^*,1)-1)\sqrt{\G(\l^*,s)}}, \q s\in[0,1].$$
Thus, the optimal control $u^*(\cd)$ is given by
$$u^*(s)= -{P(\l^*,s)\over 1+\l^*}X^*(s)= {(\a+1)e^{\a(2-s)}+(\a-1)e^{\a s}\over 1-e^{2\a}}, \q s\in[0,1],$$
where
$$\a= {\sqrt{(\l^*+1)(\l^*+16)}\over \l^*+1}\approx 3.6929.$$
\end{example}

\section{Conclusions}

We have developed a systematic approach to the constrained LQ optimal control problem
based on duality theory and approximation techniques.
The problem gives rise to a Riccati differential equation with infinite terminal value
as a result of the non-free feature of the terminal state.
It is shown that by solving the Riccati equation and an optimal parameter selection problem,
the optimal control can be represented as a feedback of the current and terminal states.
We extensively investigate the Riccati equation by a penalty method, and with the solutions
of two Riccati-type equations, we explicitly solve a parameterized LQ problem without the
integral quadratic constraints.
This allows us to determine the optimal parameter by simply calculating derivatives.
Our method also provides some alternative and useful viewpoint to study optimal control of
exactly controllable stochastic systems. Research on this topic is currently in progress.

\end{document}